\documentclass[12pt]{amsart}
\usepackage[margin=1in]{geometry} 

\usepackage[colorlinks=true,urlcolor=blue,
bookmarks=true,bookmarksopen=true,citecolor=blue
]{hyperref}

\usepackage{amsmath, amsfonts}      
\usepackage{color} 
\usepackage[all]{xy}
\usepackage{graphicx}

\theoremstyle{plain}

\newtheorem{thm}{Theorem}[subsection]
\newtheorem{cor}[thm]{Corollary}
\newtheorem{lem}[thm]{Lemma}
\newtheorem{prop}[thm]{Proposition}

\theoremstyle{definition}
\newtheorem{dfn}[thm]{Definition}

\theoremstyle{remark}
\newtheorem{rem}[thm]{Remark}

\theoremstyle{plain}

\newcommand{\R}{\mathbb{R}}
\newcommand{\Z}{\mathbb{Z}}

\newcommand{\Crit}{\textnormal{Crit\/}}

\newcommand{\fcaddress}{francois.charette8@gmail.com}

\begin{document}

\title[Floer homology of fibrations I]{Floer homology of fibrations I: Representing flow lines in Moore path spaces}
\date{\today}

\author{Fran\c{c}ois Charette}
\thanks{The author was supported by the University of Ottawa.}

\address{Fran\c{c}ois Charette, University of Ottawa} \email{\fcaddress}

%

\begin{abstract}
In the article \cite{Bar-Cor:Serre}, Barraud and Cornea enriched the Lagrangian Floer complex by adding cubical chains in the based loop space of the Lagrangian, and recovered the Leray-Serre spectral sequence of the based path space fibration, assuming that the Lagrangian is weakly exact and simply connected.  In the present article, we remove the simple connectivity assumption and adapt the construction to any Hurewicz fibration.  To this end, we introduce a stronger notion of local systems than the classical one, as a topological functor from the free path space of a manifold to the category of topological spaces.  Such functors arise naturally from a Hurewicz fibration.
\end{abstract}

\maketitle

%
%


\section{Introduction}
The goal of this article is to enrich the coefficients of Morse theory by adding singular chains in fibers above the base manifold, which are playing the role of local systems of chain complexes.  The resulting homology computes the singular homology of the total space of an associated Hurewicz fibration.  Our motivation comes from Lagrangian Floer homology, where unitary local systems were first introduced by Kontsevich in his ICM lecture \cite{Ko:ICM-HMS}.  Since then, various local systems were used in the theory, ranging from rank one local systems to coefficients in the group ring of the fundamental group of a Lagrangian, as well as higher rank local systems (see for example Abouzaid \cite{Abouzaid:nearby}, Abouzaid-Kragh \cite{Abou-Kragh:simplehomtypenearby}, Chantraine-Dimitroglou Rizell-Ghiggini-Golovko \cite{CDRGG:noncommaugcat}, Cho \cite{Cho:non-unitary}, Damian \cite{Dam:Audinsconj}, Konstantinov \cite{Konst:higherlocal}, Suarez \cite{Su:exactcob}, Sullivan \cite{SullM:kinvariants}), Zapolsky \cite{Zapol:orientations}, and many others).  We refer to Konstantinov's article for a presentation of Floer homology and the Fukaya category with higher rank local systems, for monotone Lagrangian submanifolds.  In this construction, the Floer differential is filtered and the first page of the associated spectral sequence is Morse homology with local coefficients, which,  by a result of Abouzaid \cite[Appendix B]{Abouzaid:nearby}, is isomorphic to singular homology with local coefficients.

In the present article, we restrict our attention to Morse homology and enrich the coefficients by strengthening the notion of local system, based on the following simple observation:  a fibration $\xymatrix{F \ar[r] & E \ar[r] & B}$ should be thought of as a local system of topological spaces above the base space $B$, assigning to each point of $B$ the fibre $F_b$ above it.  Classically \cite{Steenrod:local}, a local system on a space $B$ is a functor from the fundamental groupoid of $B$ to some category, which assigns to each homotopy class of paths an isomorphism in that category.  We introduce local systems of topological spaces or chain complexes, as a functor of topological categories from the free Moore path space of $B$ (paths with both endpoints allowed to move freely in $B$), which assigns to each free Moore path a weak isomorphism (these will be either homotopy equivalences of topological spaces or chain homotopy equivalences of chain complexes);  homotopic paths need not induce the same weak isomorphism.   

Our construction generalizes work of Barraud and Cornea \cite{Bar-Cor:Serre}, who enrich the coefficients in Floer homology for a weakly exact, simply connected Lagrangian, to take into account chains in the based loop space; they do so by representing moduli spaces of Floer trajectories \emph{of arbitrary dimensions} in the based loop space of a Lagrangian.  Their construction is related to the based path space fibration $\xymatrix{\Omega B \ar[r] & PB \ar[r] & B}$, since the enriched differential is filtered and the associated spectral sequence is isomorphic to the Leray-Serre spectral sequence of the fibration.  In our case, non trivial fundamental groups are allowed and we recover the Leray-Serre spectral sequence of any fibration whose base is a closed manifold.  By restricting the coefficients to zero dimensional chains in the fibre, we recover the various constructions of local systems in Floer homology mentioned above (for weakly exact Lagrangians), thus providing a more unified framework in which to use local systems.

We plan to generalize the construction to monotone Lagrangians, where bubbling issues in higher dimensional moduli spaces  should be addressed (see Barraud-Cornea \cite{Bar-Co:Quantization} for the Hamiltonian Floer homology case).  As the Floer differential of a single Lagrangian is a deformation of the Morse differential, it is necessary to first understand the simpler context of Morse theory, as we do here.

\subsection{Main results}
\begin{thm}\label{thm:main}
Given a Hurewicz fibration $\xymatrix{F \ar[r] & E \ar[r]^-\pi & B}$, where $B$ is a closed manifold, and a Morse-Smale pair $\mathcal{D}=(f, \rho)$ on $B$, there exists an enrichment of the Morse complex $\mathcal{C}(B, \mathcal{D}; \mathcal{C}_*(F))$, with local coefficients in chains in the fiber, whose homology is isomorphic to $H_*(E;\Z_2)$.  Moreover, there exists a filtration of this complex, such that the resulting spectral sequence $(E^*_{*,*}, d_*)$ satisfies:
\begin{itemize}
 \item $E^1_{p,q} \cong C_p^M(B, \mathcal{D}; \mathcal{H}_q(F;\Z_2))$.  The differential $d_1 \colon E^1_{p,q} \to E^1_{p-1, q}$ is the Morse differential with local coefficients in the homology of the fibre.
 \item $E^k$ is isomorphic to the Leray-Serre spectral sequence of the fibration $\pi\colon E \to B$, for $k \geq 2$.
\end{itemize}
\end{thm}
\begin{rem}
A similar result for fibrations both whose base and fiber are manifolds was obtained by Hutchings \cite{Hutchings:families}.  His methods make use of chains in the base and a Morse function on the fibre, but they do not involve higher dimensional moduli spaces of flow lines.
\end{rem}

By choosing specific fibrations, we recover various constructions from the litterature by restricting our enriched complex to zero dimensional chains in the fibre.  For instance, let $E_{\text{reg}}$ be the local system on $B$ with monodromy $\rho\colon \Z[\pi_1(B, b_0)] \to \text{End}(\Z[\pi_1(B, b_0)])$ induced from left multiplication of $\pi_1(B)$ on itself.  The image of $\rho$ generates a sublocal system $E_{\rho}$ of $\text{End}(\Z[\pi_1(B, b_0)])$, which is a local system of operator rings in Steenrod's terminology \cite[Section 5]{Steenrod:local} and is denoted by $\text{End}_{\text{mon}}(E_{\text{reg}})$ in \cite[Section 2.5]{Konst:higherlocal}.  This system is closely related to the free loop space fibration $\xymatrix{\Omega B \ar[r] & \Lambda B \ar[r] & B},$ although it sees only part of the second page of its Leray-Serre spectral sequence.
\begin{prop}
Let $\mathcal{E}$ be the Leray-Serre spectral sequence of the free loop space fibration of $B$ and let $E_\rho$ be the local system of operator rings defined above.  Then, the Morse homology of $B$ with local coefficients in $E_\rho$ is isomorphic to $\mathcal{E}^2_{*,0}$.
\end{prop}
\begin{prop}
In the case of the Moore based path space fibration $\xymatrix{\Omega B \to PB \to B}$, our construction recovers the complex of Barraud-Cornea \cite{Bar-Cor:Serre}.  Restricting to zero-chains $C_0(\Omega B)$, we recover Damian's lifted Morse complex \cite{Dam:Audinsconj}, computing the homology of the universal cover of $B$.
\end{prop}
These propositions are proved directly by applying Theorem \ref{thm:main} to the constructions presented in Sections \ref{sec:freeloopspace} and \ref{sec:moorepathspace}.

\subsection{Acknowledgements}
This project started ages ago during my master's thesis in 2007, under the supervision of Octav Cornea.  I would like to thank him for suggesting the topic and for many valuable discussions over this long period of time!
The paper also benefited from conversations with Baptiste Chantraine, Cl\'ement Hyvrier, Emmanuel Giroux, and Momchil Konstantinov.
\section{Notations}
Given a topological space $B$, the free Moore path space of $B$ is the set $\Pi B = \{ \gamma\colon [0,r] \to B \; | \; r \geq 0 \}$; an element of this set is sometimes written $(\gamma, r)$  to specify the length of the path.  All path spaces are endowed with the compact open topology.  There are two evaluation maps $\text{ev}_0\colon\Pi B \to B$, $\text{ev}_0(\gamma) = \gamma(0)$ and $p\colon \Pi B \to B$, $p(\gamma, r) = \gamma(r)$.  This space can be seen as a topological category with objects $\text{Ob}(\Pi B) = \{ x \in B\}$.  Morphisms from $x$ to $y$ are given by the space of Moore paths from $x$ to $y$:  $\text{Mor}(x, y) = P_{x,y} B = \{\gamma\colon [0,r] \to B \; | \; r \geq 0, \; \gamma(0)=x, \; \gamma(r) = y \}$.  Composition of two morphisms $\gamma_1, \gamma_2$ is given by concatenation of paths, written from left to right as $\gamma_1 \gamma_2$, and is associative.  Identifying all homotopic paths yields the fundamental groupoid of $B$, $\Pi_1 B$.

Given subsets $X, Y \subset B$, $P_{X, Y} B$ is the space of Moore paths starting on $X$ and ending on $Y$.  Fixing $x \in B$, $P_{x,x}$ is the space of Moore loops based at $x$, denoted also by $\Omega B_x$; it is a topological monoid, with identity element given by the constant path based at $x$.  The free Moore loop space of $B$ is $\Lambda B = \{ \gamma\colon [0,r] \to B \; | \; r\geq 0, \; \gamma(0) = \gamma(r)\}$.  The Moore path space based at $x \in B$ is defined by $P B := P_{x, B} B$.

\subsection{Local systems}
We refer the reader to Steenrod's survey paper \cite{Steenrod:local} for the definitions of local systems, systems of operator rings and homology with local coefficients.

\section{Parallel transport and chain complexes}
\subsection{Fibrations and transport}
The material presented here is taken from May's book \cite[Chapter 3]{May:classifying}.  Given a continuous map $\pi\colon E \to B$, denote by $F_x$ the fiber $\pi^{-1}(x)$.  Define $\Gamma E \subset E \times \Pi B$ as the following pullback:
$$\xymatrix{
\Gamma E \ar[r] \ar[d]^{p_2} & E \ar[d]^\pi\\
 \Pi B \ar[r]^-{\text{ev}_0} & B
}$$
It comes with a map $\eta\colon E \to \Gamma E$, $\eta(e) = (e, \pi(e))$.  A lifting function for $\pi\colon E \to B$ is a map $T\colon \Gamma E \to E$ fitting in the commutative diagrams below:
$$\xymatrix{
\Gamma E \ar[r]^-T \ar[rd]^-{\Gamma \pi} \ar[d]^{p_2} & E \ar[d]^\pi & & E \ar[r]^-\eta \ar[rd]_1 & \Gamma E \ar[d]^-T\\
 \Pi B \ar[r]^{p} & B & & & E
}
$$
A parallel transport map (called a transitive lifting function in \cite{May:classifying}) is a lifting function $T$ with the property that, given $\alpha \in \Pi B$ and $\beta \in \Pi B$ such that $p(\alpha) = \beta(0)$, we have $T(T(e,\alpha), \beta) = T(e, \alpha \beta)$.  We have:
\begin{itemize}
\item A map $\pi\colon E \to B$ is a Hurewicz fibration if and only if it admits a lifting function (\cite[Proposition 3.4]{May:classifying}).
\item If $\pi$ is a fibration, then the maps $\eta$ and $T$ are inverse fiber homotopy equivalences. (\cite[Remarks 3.6]{May:classifying})
\item The map $\mu\colon\Gamma \Gamma E \to \Gamma E$, $\mu((e, \alpha), \beta) = (e, \alpha \beta)$, is a parallel transport map for $\Gamma \pi\colon \Gamma E \to B$.
\end{itemize}
The fibre $\hat{F}_b = \Gamma \pi^{-1}(b) = \bigcup_{x \in B} F_x \times P_{x,b} B$ has the same homotopy type as $F_b$, since  there is a fibration $\xymatrix{F \ar[r] & \hat{F}_b \ar[r] & P_{B,b} B}$ with a contractible base.  Therefore, any fibration can be replaced by the homotopy equivalent fibration $\xymatrix{\hat{F} \ar[r] & \Gamma E \ar[r]^-{\Gamma \pi} & B}$ \emph{without changing the base}.  We may and will henceforth assume throughout the text that any Hurewicz fibration is endowed with a parallel transport map.

\subsection{Adjoint functor}
A transport map $T$ induces a functor of topological categories, called the adjoint of $T$:
\begin{align*}
T: \Pi B & \longrightarrow \text{Top}\\
\left\{
\begin{array}{c}
x  \in B\\
\gamma \in P_{x,y} B
\end{array}
\right\}
& \longrightarrow
\left\{
\begin{array}{c}
F_x \\
T(\gamma)\colon F_x \to F_y\\
f \mapsto T(f, \gamma)
\end{array}
\right\}.
\end{align*}
The transitivity of $T$ is necessary to get $T(\gamma_2)\circ T(\gamma_1) = T(\gamma_1 \gamma_2)$.  The following properties of the adjoint are obvious:
\begin{itemize}
\item $T(\gamma)$ is a homotopy equivalence.
\item The map $T\colon F_x \times P_{x,y} B \to F_y$ obtained by restriction is continuous.
\end{itemize}

\subsection{Local systems of chain complexes and cubes of transport maps}\label{sec:cubestransport}
Postcomposing the adjoint with the functor $Sing\colon \text{Top} \to \text{Ch}$ taking a space to its cubical chain complex over $\Z / 2 \Z =: \Z_2$, we obtain, for every fibration $\pi\colon E \to B$, a local system of chain complexes over $B$.  By a slight abuse of notation, we give the same name to the resulting functor:
\begin{align*}
T: \Pi B & \longrightarrow \text{Ch}\\
\left\{
\begin{array}{c}
x  \in B\\
\gamma \in P_{x,y} B
\end{array}
\right\}
& \longrightarrow
\left\{
\begin{array}{c}
C_*(F_x) \\
T(\gamma)_*\colon C_*(F_x) \to C_*(F_y)\\
\end{array}
\right\}.
\end{align*}
\noindent\textbf{Warning:}  Strictly speaking, a local system is a functor from the fundamental groupoid to a category, so that every homotopy of paths is mapped to an isomorphism.  In our context, we consider functors from $\Pi B$, so that every path is mapped to a weak-isomorphism.  In the case of chain complexes, these are chain homotopy equivalences.

Concatenation of paths induces a product on cubes.  Indeed, given $\sigma \in C_p(P_{x,y} B)$, $\sigma' \in C_q(P_{y,z} B)$, the product $\sigma \cdot \sigma' \in C_{p+q}(P_{x,z} B)$ is defined by $\sigma \cdot \sigma'(u,v) = \sigma(u)\cdot \sigma'(v)$, where $u \in [0,1]^p$, $v \in [0,1]^q$.  The transport map $T$, being continuous, induces chain maps:
$$
\xymatrix{
 C_p(F_x) \otimes C_q(P_{x,y} B) \ar[r]^{\times} & C_{p+q}(F_x \times P_{x,y} B) \ar[r]^-{T_*} & C_{p+q}(F_y) 
}
$$
which will be written either $T(\alpha \times \gamma)$ or $T(\gamma)(\alpha)$.  Since $T$ is a functor, given $\sigma \in C_*(P_{x,y}B), \; \sigma' \in C_*(P_{y,z} B)$ and $\alpha \in C_*(F_x)$, we get
\begin{align}\label{eq:Tcompatible}
T(\sigma \cdot \sigma')(\alpha) = T(\sigma')\circ T(\sigma)(\alpha).
\end{align}

\begin{rem}[$\infty$-categories]
The language of $\infty$-categories seems well suited to formalize the functorial properties of $T$.  Indeed, $\Pi B$, being a topological category, is an $\infty$-category (see Lurie \cite{Lurie:highertoposbook}).  On the other hand, $\text{Ch}$ is also an $\infty$-category, where 1-morphisms are chain maps, 2-morphisms are chain homotopies, and so on.  Using this language and fixing $x, y \in B$, $T$ induces the following map:
$$T\colon C_p(P_{x,y} B) \to \text{Mor}_{p+1}(C_q(F_x), C_{p+q}(F_y))$$
$$T(\sigma)(\alpha) = T(\alpha \times \sigma),$$
where $\text{Mor}_{p+1}$ are the $(p+1)$-morphisms in $\text{Ch}$.  Indeed, for a zero chain $\sigma \in C_0(P_{x,y} B)$, $T(\sigma)\colon C_q(F_x) \to C_q(F_y)$ is a chain-map admitting a quasi-inverse.  Given a 1-chain $\sigma \in C_1(P_{x,y} B)$, then $T(\sigma)\colon C_q(F_x) \to C_{q+1}(F_y)$ is a chain-homotopy between $T(\sigma(1))$ and $T(\sigma(0))$, and so on.
\end{rem}

\subsection{Genuine local systems}
By identifying homotopic morphisms in $\Pi B$ and applying the homology functor, the adjoint yields a genuine local system of vector spaces over $\Z_2$, denoted by $\mathcal{H}_*(F)$:
$$\xymatrix{
\Pi B \ar[r]^-T \ar[d]^-{\sim}& \text{Ch} \ar[d]^-{H_*}\\
 \Pi_1(B) \ar[r]^-T & \text{Vect}
}$$
Up to isomorphism, this local system is completely determined by its monodromy representation $\pi_1(B, b) \to \text{Aut}(H_*(F_b))$ and its associated local system of operator rings $\Z_2[\pi_1(B, b)] \to \text{End}_{\Z_2}(H_*(F_b))$.

\subsection{Examples}\label{sec:examples}
\subsubsection{The free Moore loop space}\label{sec:freeloopspace}
The free Moore loop space $\Lambda B$ admits the following parallel transport map, that we fix from now on:
$$T\colon \Gamma (\Lambda B) \to \Lambda B$$
$$(\alpha, \gamma) \mapsto \gamma^{-1} \alpha \gamma.$$
Here it is crucial to use Moore loops to obtain $T(T(\alpha, \gamma_1), \gamma_2) = T(\alpha, \gamma_1 \gamma_2)$.  Using paths from the unit interval would yield an equality only up to homotopy.

We compute the system of operator rings on $\text{End}_{\Z_2}(H_0(\Omega B_b))$.  As a $\Z_2$-algebra, we have $H_0(\Omega B_b) \cong \Z_2[\pi_1(B, b)]$, where the algebra structure on homology is the Pontryagin product.  Then it is clear from the definition of the transport that the map
$$\Z_2[\pi_1(B,b)] \to \text{End}(H_0(\Omega B_b)) = \text{End}(\Z_2[\pi_1(B, b)])$$
is induced from the conjugation representation $\pi_1(B, b) \to \text{Aut}(\pi_1(B,b))$.  This local system was used in Floer theory by Konstantinov in \cite[pp. 24-25]{Konst:higherlocal}


\subsubsection{The Moore path space}\label{sec:moorepathspace}
The Moore path space of $B$, based at $* \in B$, admits the following transport map:
$$T\colon \Gamma (PB) \to PB$$
$$(\alpha, \gamma) \mapsto \alpha \beta,$$
which was used by Barraud-Cornea \cite{Bar-Cor:Serre} to define a Floer complex enriched by chains in the based loop space.

The system of operator rings on $\text{End}_{\Z_2}(H_0(\Omega B_b))$ is the map
$$\Z_2[\pi_1(B,b)] \to \text{End}(H_0(\Omega B_b)) = \text{End}(\Z_2[\pi_1(B, b)])$$ induced from the left action of $\pi_1(B,b)$ on itself.  This local system has been fruitful in Lagrangian Floer theory, for example in the work of Abouzaid-Kragh \cite{Abou-Kragh:simplehomtypenearby}, Damian \cite{Dam:Audinsconj}, Lopez-Suarez \cite{Su:exactcob}, and Sullivan \cite{SullM:kinvariants}.

\subsubsection{Fibre bundles and connections}
Let $\pi\colon E \to B$ be a smooth vector bundle endowed with a connection $\mathcal{H}$.  Then $\mathcal{H}$ defines parallel transport isomorphisms along a curve $\gamma \in P_{x, y}B$, denoted by $T(\gamma)\colon E_x \to E_y$.  This map extends to $\Gamma E$ and defines a parallel transport map for $\pi\colon E \to B$. 
\section{The Morse complex enriched over chains in Moore path spaces}
In this section, we follow rather closely a construction of Barraud-Cornea \cite{Bar-Cor:Serre}.  The main difference is that we do not restrict to the Moore path space fibration, our construction is valid for any fibration whose base is a closed manifold.   Moreover, we do not impose simple connectivity assumptions, which accounts for the presence of homology with local coefficients.  

\subsection{Local coefficients}\label{sec:coeff}
Recall from \S \ref{sec:cubestransport} that there is a local system of chain complexes on $B$, by taking, for every $x$ in $B$, the complex $C_*(F_x)$.  This local system of graded chain complexes is denoted by $\mathcal{C}_*(F)$.

\subsection{Representing chain system of flow lines}
Fix $\mathcal{D}=(f, \rho)$ a Morse-Smale pair on $B$.  The set of critical points of $f$, whose index is $k$, is denoted by $\Crit_k f$.  The index of a critical point $x$ is written $|x|$.  Given two critical points $x$ and $y$, the moduli space of negative gradient flow lines connecting $x$ to $y$, modulo $\R$-translation, is $\mathcal{M}(x,y; \mathcal{D})$, or simply $\mathcal{M}(x,y)$ when the context is clear.  By the Morse-Smale condition, it is a manifold of dimension $|x| - |y| -1$, provided it is not empty.  It can be compactified into a manifold with corners $\overline{\mathcal{M}}(x,y)$ whose boundary is $\partial \overline{\mathcal{M}}(x,y) = \bigcup_{z \in \Crit f} \overline{\mathcal{M}}(x,z) \times \overline{\mathcal{M}}(z,y)$.

As each flow line is embedded in $B$, we get a continuous representation of flow lines in the Moore path space:
$$\Phi\colon \mathcal{M}(x,y; \mathcal{D}) \to P_{x,y} B$$
$$\gamma \mapsto \Phi(\gamma)\colon [0, f(x)-f(y)] \to B,$$
where $\Phi(\gamma)(t) = \gamma \cap f^{-1}(f(x)-t)$.  This map extends continuously to the compactification and, when restricted to the boundary, is compatible with composition of paths, in the sense that, given $\gamma_1 \in \overline{\mathcal{M}}(x,z)$ and $\gamma_2 \in \overline{\mathcal{M}}(z,y)$, we have $\Phi(\gamma_1, \gamma_2) = \Phi(\gamma_1)\cdot \Phi(\gamma_2) \in P_{x,y} B$.

Recall from \cite[\S 2.2.3]{Bar-Cor:Serre}, that there exist cubical chains $s^x_y \in C_{|x|-|y|-1}(\overline{\mathcal{M}}(x,y))$, called a representing chain system, satisfying the following two properties:
\begin{enumerate}
\item In the short exact sequence of the pair $(\overline{\mathcal{M}}(x,y), \partial \overline{\mathcal{M}}(x,y))$,   $s^x_y$ is mapped to a cycle in $C_{|x|-|y|-1}(\overline{\mathcal{M}}(x,y), \partial \overline{\mathcal{M}}(x,y))$ representing the fundamental class.
\item $\partial s^x_y = \sum_{z\in \Crit f} s^x_z \times s^z_y$.
\end{enumerate}
Set $a^x_y = \Phi_*(s^x_y)\in C_{|x|-|y|-1}(P_{x,y} B)$.  By the properties of $\Phi$ and the representing chain system, we get $\partial a^x_y = \sum a^x_z \cdot a^z_y$.

\subsection{The complex}
Given a Morse-Smale pair $\mathcal{D} = (f, \rho)$, we define the Morse complex of $f$ with coefficients in the local system $\mathcal{C}_*(F)$ from \S \ref{sec:coeff} by:
$$C_*(B, \mathcal{D}; \mathcal{C}_*(F)) = \bigoplus_{x \in \Crit f} C_*(F_x) \otimes x$$
The degree of a homogeneous element $\alpha \otimes x$, where $\alpha \in C_p(F_x)$, $x \in \Crit_q f$, is $p+q$.  The differential is the $\Z_2$-linear map
$$d\colon C_*(B, \mathcal{D};\mathcal{C}_*(F)) \to C_{*-1}(B, \mathcal{D}; \mathcal{C}_*(F))$$
$$d(\alpha \otimes x) = \sum_{y \in \Crit f} T(a^x_y)(\alpha) \otimes y + (\partial \alpha)\otimes x,$$
where $T(a^x_y)(\alpha)$ is the transport map, see \S \ref{sec:cubestransport}.
\begin{lem}
The map $d$ squares to zero.
\end{lem}
\begin{proof}
We compute, using (\ref{eq:Tcompatible}):
\begin{align*}
d^2(\alpha \otimes x) =& \sum_{z,y} T(a^y_z)(T(a^x_y)(\alpha)) \otimes z + 
\sum_y \partial(T(a^x_y(\alpha)))\otimes y + 
\sum_y T(a^x_y)(\partial \alpha) \otimes y + 
\partial^2 \alpha \otimes x\\
=& \sum_{z,y} T(a^x_y \cdot a^y_z)(\alpha)\otimes z + 
\sum_y T(a^x_y)(\partial \alpha)\otimes y + \sum_y T(\partial a^x_y)(\alpha)\otimes y + 
\sum_y T(a^x_y)(\partial \alpha) \otimes y\\
=& \sum_z T(\partial a^x_z)(\alpha)\otimes z+
\sum_y T(a^x_y)(\partial \alpha)\otimes y + \sum_y T(\partial a^x_y)(\alpha)\otimes y + 
\sum_y T(a^x_y)(\partial \alpha) \otimes y\\
=& \; 0 \mod 2.
\end{align*}
\end{proof}
We will prove below that the homology of this chain complex is the homology of $E$ with $\Z_2$-coefficients.

\subsubsection{A spectral sequence}\label{sec:filtration}
The Morse index of critical points defines an increasing, bounded, filtration on the complex $C_*(B, \mathcal{D}; \mathcal{C}_*(F))$ and its homology.  Set
$$F_p C = \{ \sum \epsilon_i \gamma_i \otimes x_i \in C_* \; | \; \epsilon_i \in \Z/2\Z, \; |x| \leq p \}.$$
The differential obviously preserves the filtration, hence there is a homology spectral sequence $E^*_{p,q}$ converging to $H(C_*(L, \mathcal{D}; \mathcal{C}_*(F)), d)$.  Before proving that this is the Leray-Serre spectral sequence of the fibration $\pi\colon E \to B$, let us compute the first few pages by hand, as a sanity check.

Page zero is $F_p C / F_{p-1} C$, that is $E^0_{p,q} = C_p(f; \mathcal{C}_q(F))$ and the differential induced from $d$ is the boundary operator on $C_q(F)$.  Therefore,
$E^1_{p,q} \cong C_p(f; \mathcal{H}_q(F)).$  On the first page, the induced differential of $\alpha \otimes x$ is
$d_1 (\alpha \otimes x) = \sum_{|y| = |x|-1} T(a^x_y)(\alpha) \otimes y$, which is well-defined since $\partial a^x_y =0$ for degree reasons.  Therefore, the second page is $E^2_{p,q} \cong H_p^M(B;\mathcal{H}_q(F; \Z_2))$.  By definition, this is the Morse homology of $B$ with local coefficients in $H_q(F;\Z_2)$, where the monodromy of the local system is induced from the adjoint of $T$.  Hence, the second page is isomorphic to the second page of the Leray-Serre spectral sequence of the fibration $\pi\colon E \to B$.  We wish to prove that this isomorphism is induced from a map of spectral sequences.

\subsubsection{Invariance}\label{sec:invariance}
Given two Morse-Smale pairs $\mathcal{D}_i=(f_i, \rho_i)$, one can modify the usual proof of invariance of Morse homology to our context, using the transport map $T$ as well as a generating chain system for the moduli space of gradient flow lines of the homotopy.  As this is similar to the construction in the previous section and is an adaptation of \cite[Theorem 2.10 a.]{Bar-Cor:Serre}, we omit the proof.
\begin{prop}
A generic choice of homotopy between $\mathcal{D}_1$ and $\mathcal{D}_2$ induces a chain-map $$\Phi\colon C_*(B, \mathcal{D}_1; \mathcal{C}_*(F)) \to C_*(B, \mathcal{D}_2; \mathcal{C}_*(F))$$
which respects the filtrations and induces an isomorphism from the respective spectral sequences.
\end{prop}
\section{Comparison with the Leray-Serre spectral sequence}
\subsection{A survey of Serre's proof}
In order to compare the spectral sequence from \S \ref{sec:filtration} with the Leray-Serre spectral sequence, we first recall some details of Serre's proof \cite[Chapitre 2, no. 4,5,6]{Serre:fibration}.

Let $\xymatrix{
F \ar[r] &  E  \ar[r]^\pi & B
}
$
be a Hurewicz fibration with a parallel transport map $T$.  Let $T_{p,q} \subset C_{p+q}(E)$ be the cubical chains of degree $p+q$ in $E$ generated by cubes $u$ with the property that $\pi \circ u$ does not depend on its first $q$ coordinates.  Define $T_p = \oplus_q T_{p,q}$.  This gives $C_*(E)$ the structure of a filtered differential graded module, and the resulting spectral sequence is denoted by $\mathcal{E}^*_{p,q}$.  

Given a generator $u \in T_{p,q}$, define $Bu \in C_p(B)$ and $Fu \in C_q(F)$ by
$$Bu(t_1, \dots, t_p) = \pi\circ u(y_1, \dots y_q, t_1, \dots t_q), \quad y_i \text{ arbitrary},$$
$$Fu(t_1, \dots t_q) = u(t_1, \dots, t_q, 0,0, \dots,0).$$
By definition of $T_{p,q}$, these two operations are well defined and $Fu$ lies in the fiber above $\pi \circ u(0, 0, \dots, 0)$.  Set $J_p = C_p(B;\mathcal{C}(F))$ with the differential $d(b \otimes f) = b \otimes \partial f$.  Then Serre's results show that 
$$B \otimes F\colon \mathcal{E}^0_{*,*} \to J_{*,*}$$
$$\sigma \mapsto B\sigma \otimes F\sigma$$
induces a chain-homotopy equivalence.  Moreover, the induced map on the first page of the spectral sequences is a quasi-isomorphism which fits in the following commutative diagram:
\begin{equation}\label{eq:diagSerre}
\xymatrix{
   \mathcal{E}^1_{p,q} \ar[r]^-{B \otimes F}_-{\simeq} \ar[d]^{d_1} &  C_p(B; \mathcal{H}_q(F)) \ar[d]^{\partial_\text{loc}} \\
   \mathcal{E}^1_{p-1,q} \ar[r]^-{B \otimes F}_-{\simeq} &  C_{p-1}(B; \mathcal{H}_q(F))
}
\end{equation}
where $\partial_\text{loc}$ is the singular differential with coefficients in the local system $\mathcal{H}_*(F)$ induced from the fibration.

\subsection{Construction of the isomorphism}
By the result of \S \ref{sec:invariance}, the spectral sequence does not depend on the choice of Morse-Smale pair.  Therefore, take a Morse-Smale pair $\mathcal{D}=(f, \rho)$, where $f$ is a self-indexed Morse function with a unique minimum $m$.  Set $B^t = f^{-1}((-\infty, t])$, where $t \in \R$.  Barraud-Cornea, in \cite[\S 2.4.6]{Bar-Cor:Serre}, introduce the notion of the blow-up of unstable manifolds in order to compare their enriched Morse complex to the Leray-Serre spectral sequence of the path space fibration (similar constructions were also considered by Hutchings-Lee \cite{HutLee:torsion}).  Let us recall some properties of these manifolds, which we give without proof.
\begin{dfn}
Given $x \in \Crit f$, the blow-up of the unstable manifold of $x$ is the manifold with boundary
$$\widehat{\mathcal{M}}(x) = \overline{\mathcal{M}}(x, m)\times [0, f(x)] / \sim$$
where $\sim$ is an equivalence relation which we will not explicitly need.  It satisfies the following properties:
\begin{itemize}
\item It is homeomorphic to a disc of dimension $|x|$.
\item Its boundary satisfies $$\partial \widehat{\mathcal{M}}(x) = \bigcup_y \overline{\mathcal{M}}(x, y)\times \widehat{\mathcal{M}}(y).$$
\item It admits a system of generating chain $\lambda_x \in C_{|x|}(\widehat{\mathcal{M}}(x))$ such that $\partial \lambda_x = \sum_y s^x_y\times \lambda_y$.
\item There is a map $o\colon \widehat{\mathcal{M}}(x) \to B$ induced from $\overline{\mathcal{M}(x, m)}\times [0, f(x)] \to B$ given by intersecting a flow line at times $t \in [0, f(x)]$ with the hypersurface $f^{-1}(t)$.  Its image coincides with the compactification of the unstable manifold of $x$ and gives a cellular decomposition of $B$, with attaching map being $o_{| \partial \widehat{\mathcal{M}}(x)}$.
\item The blow-up has a basepoint $[x]$ such that $o([x]) = x$ and there is a continuous representation 
$\beta\colon \widehat{\mathcal{M}}(x) \to P(\widehat{\mathcal{M}}(x))$ with the property that the image of the composition $\xymatrix{\widehat{\mathcal{M}}(x) \ar[r]^-{\beta} & P(\widehat{\mathcal{M}}(x)) \ar[r]^-o & PB}$ lies in 
$P_{x, B^{|x|}}B$.  The image of this map consists of all flow lines starting at $x$ and ending on some point of the compactified unstable manifold of $x$.
\item The map $o \circ \beta = \beta'$ fits in the commutative diagram:
$$\xymatrix{
\overline{\mathcal{M}}(x,y)\times \widehat{\mathcal{M}}(y) \ar[r]^{\Phi \times \beta'}  \ar[dr]^{\beta'} & P_{x, y}B \times P_{y, B^{|y|}}B \ar[d]^{\cdot}\\
 & P_{x, B^{|x|}}B
}.
$$
\item On the chain level, we have $\beta'_*(s^x_y \times \lambda_y) = a^x_y \cdot \beta'_*(\lambda_y)$.
\end{itemize}
\end{dfn}
Using the blow-up construction, we define a comparison map
$$\Psi\colon C_*(B, \mathcal{D}; \mathcal{C}_*(F)) \to C_*(E)$$
$$\gamma\otimes x \mapsto T_*(\beta'_*(\lambda_x))(\gamma) = T_*(\gamma \times \beta'_*(\lambda_x)).$$
By definition, we have $\Psi(F_p C_{p+q}) \subset T_{p,q}$.
\begin{lem}
The map $\Psi$ is a chain-map.
\end{lem}
\begin{proof}
\begin{align*}
\partial T_*(\beta'_*(\lambda_x))(\gamma) &= T_*(\beta'_*(\sum_y s^x_y\times \lambda_y)) + T_*(\beta'_*(\lambda_x))(\partial \gamma)\\
&= \sum_y T_*(a^x_y\cdot \beta'_*(\lambda_y))(\gamma) + T_*(\beta'_*(\lambda_x))(\partial \gamma)\\
&= \sum_y T_*(\beta'_*(\lambda_y))\circ T_*(a^x_y)(\gamma) + T_*(\beta'_*(\lambda_x))(\partial \gamma)\\
&=\Psi(d(\gamma \otimes x)).
\end{align*}
\end{proof}
\begin{thm}\label{thm:isotoLeraySerre}
The map $\Psi^2_{p,q}\colon E^2_{p,q} \to \mathcal{E}^2_{p,q} $ induced from $\Psi$ on the second page of the respective spectral sequences, is an isomorphism.
\end{thm}
Since the filtrations defining the spectral sequences are bounded, we obtain:
\begin{cor}
The homology of $(C(B; \mathcal{C}_*(F)), d)$ is isomorphic to the homology of the total space $E$.
\end{cor}
\begin{proof}
The map $\Psi$ fits in the following commutative diagram (see also (\ref{eq:diagSerre})):
$$
\xymatrix{
  C_p^M(B; \mathcal{H}_q(F)) \ar[r]^-{\Psi^1} \ar[d]^{d_1}& \mathcal{E}^1_{p,q} \ar[r]^-{B \otimes F}  \ar[d]^{d_1} &  C_p(B; \mathcal{H}_q(F)) \ar[d]^{\partial_\text{loc}} \\
   C_{p-1}^M(B; \mathcal{H}_q(F)) \ar[r]^-{\Psi^1} & \mathcal{E}^1_{p-1,q} \ar[r]^-{B \otimes F}&  C_{p-1}(B; \mathcal{H}_q(F))
}.
$$
Recalling the definitions of the maps $B$ and $F$, we have:
\begin{align*}
B(\Psi(\gamma \otimes x)) = B(T_*(\gamma \times \beta'(\lambda_x))) = \pi_* T_*(\gamma \times \beta'(\lambda_x)) =p_* \beta'(\lambda_x) = o_*(\lambda_x).
\end{align*}
We have used the following diagram in the third equality above:
$$
\xymatrix{
F_x \times P_{x, B^{|x|}} B \ar[r]^-{T} \ar[d]^{p_2} & E \ar[d]^\pi\\
P_{x, B^{|x|}} B  \ar[r]^-{p} & B
}
$$
where $p_2$ is the projection to the second factor and $p$ is the endpoint map.  Moreover,
\begin{align*}
F(\Psi(\gamma \otimes x))(t_1, \dots, t_q) = T_*(\gamma(t_1, \dots, t_q) \times \beta'(\lambda_x(0, \dots, 0))).
\end{align*}
Writing $\lambda_x = \sum_i u_i$, we have that $\beta'(\lambda_x(0, \dots, 0))$ is a finite set of paths in $B$, one for each $i$, from $x$ to the leading vertex of the cube $u_i$.  Since $o_*(\lambda_x)$ is mapped to the fundamental class of the cell corresponding to $x$, which is contractible, transport is independent of the path chosen and $B \times F(\Psi(\gamma \otimes x)) $ is equal to $o_*(\lambda_x) \otimes \gamma$, a generator of the cellular complex of $B$ with coefficients in the local system $\mathcal{H}(F)$.  The conclusion follows, since the homology of this cellular complex is isomorphic to cubical homology with local coefficients.  Compare also with \cite[Appendix B]{Abouzaid:nearby}, where it is shown that Morse homology with local coefficients is isomomorphic to singular homology with local coefficients.
\end{proof}

\bibliographystyle{alpha}
\bibliography{../../biblio_latex/bibliography}

\def\cprime{$'$} \def\cprime{$'$}
\begin{thebibliography}{CDRGG16}

\bibitem[Abo12]{Abouzaid:nearby}
Mohammed Abouzaid.
\newblock Nearby {L}agrangians with vanishing {M}aslov class are homotopy
  equivalent.
\newblock {\em Invent. Math.}, 189(2):251--313, 2012.

\bibitem[AK16]{Abou-Kragh:simplehomtypenearby}
M.~Abouzaid and T.~Kragh.
\newblock Simple homotopy equivalence of nearby {L}agrangians.
\newblock Preprint, can be found at \url{}https://arxiv.org/abs/1603.05431,
  2016.

\bibitem[BC07a]{Bar-Cor:Serre}
J.-F. Barraud and O.~Cornea.
\newblock Lagrangian intersections and the {S}erre spectral sequence.
\newblock {\em Annals of Mathematics}, 166:657--722, 2007.

\bibitem[BC07b]{Bar-Co:Quantization}
Jean-Francois Barraud and Octav Cornea.
\newblock Quantization of the {S}erre spectral sequence.
\newblock {\em J. Symplectic Geom.}, 5(3):249--280, 2007.

\bibitem[CDRGG16]{CDRGG:noncommaugcat}
Baptiste Chantra\^{i}ne, Georgios Dimitroglou~Rizell, Paolo Ghiggini, and Roman
  Golovko.
\newblock Noncommutative augmentation categories.
\newblock Preprint. Available at \url{https://arxiv.org/abs/1603.04253}, 2016.

\bibitem[Cho08]{Cho:non-unitary}
C.-H Cho.
\newblock Non-displaceable {L}agrangian submanifolds and {F}loer cohomology
  with non-unitary line bundle.
\newblock {\em J. Geom. Phys.}, 58(11):1465--1476, 2008.

\bibitem[Dam12]{Dam:Audinsconj}
Mihai Damian.
\newblock Floer homology on the universal cover, {A}udin's conjecture and other
  constraints on {L}agrangian submanifolds.
\newblock {\em Comment. Math. Helv.}, 87(2):433--462, 2012.

\bibitem[HL99]{HutLee:torsion}
M.~Hutchings and Y.-J. Lee.
\newblock Circle-valued {M}orse theory, {R}eidemeister torsion, and
  {S}eiberg-{W}itten invariants of {$3$}-manifolds.
\newblock {\em Topology}, 38(4):861--888, 1999.

\bibitem[Hut08]{Hutchings:families}
Michael Hutchings.
\newblock Floer homology of families. {I}.
\newblock {\em Algebr. Geom. Topol.}, 8(1):435--492, 2008.

\bibitem[Kon95]{Ko:ICM-HMS}
M.~Kontsevich.
\newblock Homological algebra of mirror symmetry.
\newblock In {\em Proceedings of the {I}nternational {C}ongress of
  {M}athematicians, {V}ol.\ 1, 2 ({Z}\"urich, 1994)}, pages 120--139, Basel,
  1995. Birkh\"auser.
\newblock arXiv:alg-geom/9411018.

\bibitem[Kon17]{Konst:higherlocal}
Momchil Konstantinov.
\newblock Higher rank local systems in {L}agrangian {F}loer theory.
\newblock Preprint. Available at \url{https://arxiv.org/abs/1701.03624}, 2017.

\bibitem[Lur09]{Lurie:highertoposbook}
Jacob Lurie.
\newblock {\em Higher topos theory}, volume 170 of {\em Annals of Mathematics
  Studies}.
\newblock Princeton University Press, Princeton, NJ, 2009.

\bibitem[May75]{May:classifying}
J.~Peter May.
\newblock Classifying spaces and fibrations.
\newblock {\em Mem. Amer. Math. Soc.}, 1(1, 155):xiii+98, 1975.

\bibitem[{S}er51]{Serre:fibration}
Jean-Pierre {S}erre.
\newblock Homologie singuli\`{e}re des espaces fibr\'es.
\newblock {\em Annals of Mathematics}, 54(3):425--505, 1951.

\bibitem[Ste43]{Steenrod:local}
N.~E. Steenrod.
\newblock Homology with local coefficients.
\newblock {\em Ann. of Math. (2)}, 44:610--627, 1943.

\bibitem[Sua14]{Su:exactcob}
L.S. Suarez.
\newblock Exact {L}agrangian cobordism and pseudo-isotopy.
\newblock Preprint, can be found at \url{http://arxiv.org/abs/1412.0697}, 2014.

\bibitem[Sul02]{SullM:kinvariants}
M.~G. Sullivan.
\newblock {$K$}-theoretic invariants for {F}loer homology.
\newblock {\em Geom. Funct. Anal.}, 12(4):810--872, 2002.

\bibitem[{Zap}15]{Zapol:orientations}
F.~{Zapolsky}.
\newblock {The Lagrangian Floer-quantum-PSS package and canonical orientations
  in Floer theory}.
\newblock {\em ArXiv e-prints 1507.02253}, July 2015.

\end{thebibliography}
\end{document}